\documentclass{amsart}

\usepackage[utf8]{inputenc}
\usepackage{hyperref}

\numberwithin{equation}{section}

\DeclareMathOperator{\Tr}{Tr}

\newtheorem{thm}{Theorem}
\numberwithin{thm}{section}

\newtheorem{lem}[thm]{Lemma}
\newtheorem{coro}[thm]{Corollary}

\theoremstyle{definition}

\newcommand{\inner}[3][]{\langle#2,#3\rangle_{#1}}

\newcommand{\norm}[2][]{\|{#2}\|_{{#1}}}
\newcommand{\Norm}[2][]{\Big\|{#2}\Big\|_{{#1}}}
\newcommand{\snorm}[2][]{|{#2}|_{{#1}}}

\newcommand{\abs}[1]{|#1|}
\newcommand{\Abs}[1]{\Big|#1\Big|}
\newcommand{\DD}{{\mathcal D}}

\newcommand{\EE}{{\mathbf E}}
\newcommand{\dd}{ { \mathrm{d}} }
\newcommand{\ee}{ { \mathrm{e}} }

\newcommand{\pt}{\partial}
\newcommand{\cT}{{ \mathcal T}}

\newcommand{\cF}{{ \mathcal F}}
\newcommand{\cD}{{ \mathcal D}}

\newcommand{\IR}{{ \mathbf R}}

\newcommand{\IP}{{ \mathbf P}}

\newcommand{\IE}{{ \mathbf E}}
\newcommand{\HS}{\mathrm{HS}}
\newcommand{\subsec}{\S~}

\title[Approximation of the Cahn-Hilliard-Cook equation] {Finite
  element approximation of the Cahn-Hilliard-Cook equation}

\author[M.~Kov\'acs]{Mih\'aly Kov\'acs$^1$}

\address{Department of Mathematics and Statistics,
  University of Otago, P.O.~Box 56, Dune\-din, New Zealand}

\email{mkovacs@maths.otago.ac.nz}
\urladdr{http://www.maths.otago.ac.nz/}

\author[S.~Larsson]{Stig Larsson$^{2}$}

\address{
  Department of Mathematical Sciences,
  Chalmers University of Technology and University of Gothenburg,
  SE--412 96 Gothenburg,
  Sweden}

\email{stig@chalmers.se}
\urladdr{http://www.math.chalmers.se/~stig}

\author[A.~Mesforush]{Ali Mesforush}

\address{
  School of Mathematical Sciences,
  Shahrood University of Technology,
  Shahrood, Iran}

\email{ali.mesforush@alumni.chalmers.se}
\urladdr{http://shahroodut.ac.ir/as/?id=S138}

\thanks{$^1$Supported by a University of Otago Research Grant (UORG)}
\thanks{$^2$Supported by the Swedish Research Council (VR) and by the
  Swedish Foundation for Strategic Research (SSF) through GMMC, the
  Gothenburg Mathematical Modelling Centre.}

\keywords{Cahn-Hilliard-Cook equation, additive noise, Wiener process,
  existence, regularity, finite element, error estimate, strong convergence}

\subjclass[2000]{65M60, 60H15, 60H35, 65C30}
\date{\today}

\begin{document}

\begin{abstract}
  We study the nonlinear stochastic Cahn-Hilliard equation perturbed by
  additive colored noise.  We show almost sure existence and
  regularity of solutions.  We introduce spatial approximation by a
  standard finite element method and prove error estimates of optimal
  order on sets of probability arbitrarily close to $1$.  We also
  prove strong convergence without known rate.
\end{abstract}

\maketitle

\section{Introduction}   \label{chc2:introduction}

We study the Cahn-Hilliard equation perturbed by noise, also known as
the Cahn-Hilliard-Cook equation (cf.~\cite{Dirk1,cook}),
\begin{equation*}
\begin{aligned}
& \dd u - \Delta w \,\dd t
= \dd W && \text{in } \ \cD\times (0,T],\\
& w =- \Delta u +f(u)   && \text{in } \ \cD\times (0,T],\\
& \frac{\pt u}{\pt n}
= \frac{\pt w}{\pt n} = 0  && \text{on } \ \pt \cD\times (0,T],\\
& u(0) = u_0  && \text{in } \ \cD.
\end{aligned}
\end{equation*}
Here $\cD$ is a bounded domain in $\IR^d$, $d=1,2,3$, and
$f(s)=s^3-s$.  Using the framework of \cite{DaPratoZabczyk} we write
this as an abstract evolution equation of the form
\begin{equation}\label{C-H-C-X-1}
\dd X + \big( A^2 X + Af(X) \big)\,\dd t = \dd W,\quad t\in (0,T];
\quad X(0) = X_0,
\end{equation}
where $A$ denotes the Neumann Laplacian considered as an unbounded
operator in the Hilbert space $H=L_2(\cD)$ and $W$ is a $Q$-Wiener
process in $H$ with respect to a filtered probability space
$(\Omega,\cF,\mathbf{P},\lbrace\cF_t\rbrace_{t \ge 0})$. We also write
$H^s=H^s(\cD)$ for the standard Sobolev spaces. See Section
\ref{chc2:preliminaries} for details.

Our goal is to study the convergence properties of the spatially
semidiscrete finite element approximation $X_h$ of $X$, which is
defined by an equation of the form
\begin{equation*}
\dd X_h + \big( A_h^2 X_h + A_h P_h f(X_h) \big)\,\dd t
= P_h \,\dd W,\quad t\in (0,T];\quad X_h(0) = P_h X_0.
\end{equation*}
In order to do so, we need to prove existence and regularity for
solutions of \eqref{C-H-C-X-1}.  Such results were first proved in
\cite{Prato}.  Under the assumption that the covariance operator $Q=I$
(space-time white noise, cylindrical noise) it was shown that there is
a process which belongs to $C([0,T],H^{-1})$ almost surely and
which is the unique solution of \eqref{C-H-C-X-1}.  Under the stronger
assumption that $A$ and $Q$ commute and that
$\Tr(A^{\delta-1}Q)<\infty$ for some $\delta>0$ (colored noise) it was
shown that the solution belongs to $C([0,T],H)$ almost surely.  Such
regularity is insufficient for proving convergence of a numerical
solution.  Our first aim is therefore to prove existence of a solution
in $C([0,T],H^{\beta})$ almost surely for some $\beta>0$.

Following the semigroup approach of \cite{DaPratoZabczyk} we write
the equation \eqref{C-H-C-X-1} as the integral equation (mild
solution)
\begin{align*}
  \begin{split}
X(t) &=  \ee^{-tA^2}X_0 - \int_0^t A\ee^{-(t-s)A^2}f(X(s))\,\dd s
+ \int_0^t \ee^{-(t-s)A^2}\,\dd W(s) \\& = Y(t) + W_A(t),
\end{split}
\end{align*}
where $\ee^{-tA^2}$ is the analytic semigroup generated by $-A^2$
(see Corollary \ref{mildxy}).
This naturally splits the solution as $X=Y+W_A$, where $W_A(t)=
\int_0^t \ee^{-(t - s)A^2}\,\dd W(s)$ is a stochastic convolution.
This convolution, and its finite element approximation, was studied in
\cite{Ali}.  In particular, it was shown there that if
$\norm[\HS]{A^{\frac{\beta-2}{2}}Q^{\frac{1}{2}}}^2<\infty$
for some $\beta\ge0$, then we have regularity of order $\beta$ in a
mean square sense; that is,
\begin{align}  \label{chc2:eq1}
  \EE\big[ \norm[H^\beta]{W_A(t)}^2\big]
  \le \norm[\HS]{A^{\frac{\beta-2}{2}}Q^{\frac{1}{2}}}^2 , \quad t\ge 0.
\end{align}
The other part, $Y$, solves a differential equation with random coefficient,
\begin{equation}\label{tran-Cahn-Hill}
 \dot{Y} + A^2 Y + A f(Y+W_A) = 0,\quad t >0; \quad
 Y(0) =X_0.
\end{equation}
This can be solved once $W_A$ is known.  This approach was also used
in \cite{Prato}, but while they used Galerkin's method and energy
estimates to solve \eqref{tran-Cahn-Hill}, we use a semigroup approach
similar to that of \cite{Stig}.  However, published results for the
deterministic Cahn-Hilliard equation do not apply directly due to the
limited regularity in \eqref{tran-Cahn-Hill}.

The nonlinear term is only locally Lipschitz and we need to control
the Lipschitz constant.  In the deterministic case studied in
\cite{Stig} this is achieved by the Lyapunov functional
\begin{equation*}
J(u) = \frac{1}{2} \norm{\nabla u}^2 + \int_{\cD} F(u) \, \dd x,
\quad u \in H^1; \quad F(s)=\tfrac14 s^4-\tfrac12 s^2,
\end{equation*}
which is nonincreasing along paths, so that $\norm[H^1]{X(t)}\le C$
for $t\ge0$.  Due to the stochastic perturbation, this is not true for
the stochastic equation \eqref{C-H-C-X-1}.  However, it is possible
find a bound for the growth of the expected value of $J(X(t))$,
\begin{align}\label{chc2:eq2}
  \IE[J(X(t))]\le C(t), \quad t\ge0.
\end{align}
This was shown in \cite{Prato} under the assumption that
$A$ and $Q$ commute and
\begin{equation}\label{chc:2AQ}
\Tr(AQ)<\infty,
\end{equation}
which is consistent with $\beta=3$ in \eqref{chc2:eq1}, since
$\norm[\HS]{A^{\frac12}Q^{\frac12}}^2=\Tr(AQ)$ in this case.  (More
generally: if $AQ$ is nuclear, then
$\|A^{\frac{1}{2}}Q^{\frac{1}{2}}\|_{\HS}^2=\Tr(AQ)$, see
\cite[Theorem 2.1]{KLLweak}.)  We repeat this in
Theorem~\ref{lyap-thm} with several improvements.  First of all we
reduce the growth of the bound from exponential to quadratic with
respect to $t$.  We also relax the assumptions: we do not assume that
$A$ and $Q$ commute; that is, have a common eigenbasis, and we do not
assume that the eigenbasis of $Q$ consists of bounded
functions. Moreover, we prove the same bound for the finite element
solution $X_h$.  Even if $A$ and $Q$ commute, this will not be true
for the corresponding finite element approximations $A_h$ and $Q_h$,
so the relaxation of this assumption is necessary for the proof of the
bound for $X_h$.

In Corollary \ref{coro2.5} we improve \eqref{chc2:eq2} to a uniform
norm bound
\begin{align*}
 \EE\Big[ \sup_{s\in[0,T]}\big(\norm[H^1]{X(s)}^2
  +\norm[H^1]{X_h(s)}^2\big)\Big]\le K_T.
\end{align*}
By means of Chebyshev's inequality we may then show that, for each
$T>0$ and $\epsilon\in(0,1)$, there are $K_T$ and
$\Omega_\epsilon\subset \Omega$ with $\IP(\Omega_\epsilon)\ge
1-\epsilon$ and such that
\begin{align*}
    \norm[H^1]{X(t)}^2+\norm[H^1]{X_h(t)}^2\le \epsilon^{-1}K_T
  \quad \text{on $\Omega_\epsilon$}, \ t\in[0,T].
\end{align*}
This bound controls the Lipschitz constant of the nonlinear term and
we show that $X\in C([0,T],H^{3})$ for $\omega\in \Omega_\epsilon$
under the assumption $\norm[\HS]{A^{\frac12}Q^{\frac12}}<\infty$,
which is consistent with \eqref{chc:2AQ} (see Theorem \ref{Ynorm3}).  We
also obtain an error estimate (see Theorem \ref{thm5.3})
\begin{align*}
     \norm{X_h(t)-X(t)}\le C(\epsilon^{-1}K_T,T) h^2|\log(h)|
 \quad \text{on $\Omega_\epsilon$}, \ t\in[0,T].
\end{align*}
The constant grows rapidly with $\epsilon^{-1}K_T$, but nevertheless
we may use this to show strong convergence (see Theorem \ref{lastthm}),
\begin{align*}
     \max_{t\in[0,T]}\EE\big[\norm{X_h(t)-X(t)}^2\big]\to 0
 \quad  \text{as $h\to0$}.
\end{align*}
To prove strong convergence with an estimate of the rate remains a
challenge for future work.  In this connection we note that even for
numerical methods for stochastic ordinary differential equations with
local Lipschitz nonlinearity there are few results on convergence
rates (cf.~\cite{HighamMaoStuart02}).

Numerical methods for the deterministic Cahn-Hilliard equation are
well covered in the literature.  There are few studies of numerical
methods for the Cahn-Hilliard-Cook equation. We are only aware of
\cite{Weber} in which convergence in probability was proved for a
difference scheme for the nonlinear equation in multiple dimensions.
For the linear equation there is \cite{KossiorisZouraris}, where
strong convergence estimates were proved for the finite element method
for the linear equation in 1-D, and the already mentioned work
\cite{Ali} on the finite element method for the stochastic convolution
in multiple dimensions.

\section{Preliminaries} \label{chc2:preliminaries}
\subsection{Norms}\label{norms}
Let $\cD \subset \IR^d$, $d =1,2,3$, be a bounded convex domain with
polygonal boundary $\pt \cD$.  Let $H = L_2(\cD)$ with standard inner
product $\inner{\cdot}{\cdot}$ and norm $\norm{\cdot}$, and
\begin{equation*}
\dot{H} = \Big \lbrace v \in H : \int_\cD v \,\dd x = 0 \Big \rbrace.
\end{equation*}
Let $P \colon H \to \dot{H}$ define the orthogonal projector. Then
\begin{equation*}
(I-P)v = \abs{\cD}^{-1}\int_\cD v \,\dd x,
\end{equation*}
is the average of $v$.  We also denote by $H^k = H^k(\cD)$ the
standard Sobolev space. We define $A = -\Delta$ with domain of
definition
\begin{equation*}
D(A) = \Big \lbrace v \in H^2: \frac{\pt v}{\pt n} = 0\
\text{on}\ \pt \cD \Big \rbrace.
\end{equation*}
Then $A$ is a positive definite, selfadjoint, unbounded, linear
operator on $\dot{H}$ with compact inverse.  When extended to $H$ as
$Av=APv$ it has an orthonormal eigenbasis $\lbrace \varphi_j
\rbrace_{j=0}^\infty$ with corresponding eigenvalues $\lbrace
\lambda_j \rbrace_{j=0}^\infty$ such that
\begin{equation*}
0 = \lambda_0 < \lambda_1 \le \lambda_2 \le \cdots \le \lambda_j
\le \cdots ,
\quad \lambda_j \to \infty.
\end{equation*}
The first eigenfunction is constant, $\varphi_0 = \abs{\cD}^{-\frac{1}{2}}$.

We define seminorms and norms
\begin{align}
\label{eq:dotnorm}
& \snorm[\alpha]{v} = \Big( \sum_{j=1}^\infty
\lambda_j^\alpha \abs{\inner{v}{\varphi_j}}^2 \Big)^{\frac{1}{2}},
\quad \alpha \in\IR,\\
\label{eq:nondotnorm}
& \norm[\alpha]{v} = \big( \snorm[\alpha]{v}^2 +
  \abs{\inner{v}{\varphi_0}}^2\big)^{\frac{1}{2}} ,
\quad \alpha \in\IR,
\end{align}
and corresponding spaces
\begin{equation*}
\dot{H}^\alpha = D(A^{\frac{\alpha}{2}})
= \Big\lbrace v\in \dot{H}: \snorm[\alpha]{v} < \infty \Big\rbrace,
\quad H^\alpha = \Big \lbrace v\in H: \norm[\alpha]{v} < \infty \Big \rbrace.
\end{equation*}
For integer order $\alpha=k\ge0$, $H^k$ coincides with the standard Sobolev spaces with
$\norm[k]{\cdot}$ equivalent to the standard norm $\norm[H^k]{\cdot}$. For example,
\begin{equation}  \label{xx}
\norm[1]{v}^2 = \snorm[1]{v}^2 + \abs{\inner{v}{\varphi_0}}^2
=\norm{\nabla v}^2 + \abs{\inner{v}{\varphi_0}}^2
\end{equation}
is equivalent to the standard norm $\norm[H^1]{v}^2$ by the Poincar\'e
inequality.

\subsection{The semigroup}
The operator $-A^2$ is the infinitesimal generator of an analytic
semigroup $\ee^{-tA^2}$ on $H$,
\begin{equation*}
\begin{split}
 \ee^{-tA^2} v
& = \sum_{j=0}^\infty \ee^{-t \lambda_j^2}\inner{v}{\varphi_j}\varphi_j
= \sum_{j=1}^\infty \ee^{-t \lambda_j^2}\inner{v}{\varphi_j}\varphi_j
 + \inner{v}{\varphi_0}\varphi_0\\
& = \ee^{-tA^2}Pv + (I-P)v.
\end{split}
\end{equation*}
The analyticity implies that
\begin{equation}\label{eq:analytic}
\norm{A^\alpha \ee^{-tA^2}v} \le Ct^{-\frac{\alpha}{2}} \ee^{-ct}
\norm{v},\quad v\in H,\ \alpha > 0.
\end{equation}

\subsection{The finite element method}  \label{subsec:fem}
Let $\lbrace \cT_h \rbrace_{h>0}$ denote a family of regular
triangulations of $\cD$ with maximal mesh size $h$.  Let $S_h$ be the
space of continuous functions on $\cD$, which are piecewise
polynomials of degree $\le1$ with respect to $\cT_h$. Hence, $S_h
\subset H^1$. We also define $\dot{S}_h = PS_h$; that is,
\begin{equation*}
\dot{S}_h = \Big \lbrace v_h \in S_h: \int_\cD v_h \,\dd x = 0 \Big \rbrace.
\end{equation*}
The space $\dot{S}_h$ is introduced only for the purpose of theory but not
for computation. Now we define the ''discrete Laplacian'' $A_h \colon
S_h \to \dot{S}_h$ by
\begin{equation*}
\inner{A_h v_h}{w_h} = \inner{\nabla v_h}{\nabla w_h},
\quad \forall v_h \in S_h,\, w_h \in \dot{S}_h.
\end{equation*}
We note that
\begin{equation}\label{chc2:a1}
\snorm[1]{v_h} = \norm{A^{\frac{1}{2}} v_h}
= \norm{\nabla v_h} = \norm{A_h^{\frac{1}{2}}v_h}, \quad v_h \in S_h.
\end{equation}
The operator $A_h$ is selfadjoint, positive definite on $\dot{S}_h$,
positive semidefinite on $S_h$, and $A_h$ has an orthonormal
eigenbasis $\lbrace \varphi_{h,j} \rbrace_{j=0}^{N_h}$ with
corresponding eigenvalues $\lbrace \lambda_{h,j} \rbrace_{j=0}^{N_h}$.
We have
\begin{equation*}
0 = \lambda_{h,0} < \lambda_{h,1} \le \cdots \le \lambda_{h,j} \le
\cdots \le \lambda_{h,N_h},
\end{equation*}
and $\varphi_{h,0} = \varphi_0 = \abs{\cD}^{-\frac{1}{2}}$.
Moreover, we define $\ee^{-tA_h^2} \colon S_h \to S_h$ by
  \begin{equation*}
\begin{split}
  \ee^{-t A_h^2} v_h
  = \sum_{j=0}^{N_h} \ee^{-t \lambda_{h,j}}
   \inner{v_h}{\varphi_{h,j}}\varphi_{h,j}
  = \sum_{j=1}^{N_h} \ee^{-t\lambda_{h,j}}
   \inner{v_h}{\varphi_{h,j}}\varphi_{h,j}
   + \inner{v_h}{\varphi_0}\varphi_{0},
  \end{split}
\end{equation*}
and the orthogonal projector $P_h \colon H \to S_h$ by
\begin{equation}\label{P_h}
\inner{P_h v}{w_h} = \inner{v}{w_h}\quad \forall v \in H,\, w_h \in S_h.
\end{equation}
Clearly, $P_h \colon \dot{H} \to \dot{S}_h$ and
\begin{equation*}
\ee^{-t A_h^2}P_h v = \ee^{-t A_h^2}P_h P v + (I-P)v.
\end{equation*}
We have a discrete analog of \eqref{eq:analytic},
\begin{equation}\label{eq:analytich}
\norm{A_h^\alpha\ee^{-tA_h^2}v_h} \le Ct^{-\frac{\alpha}{2}} \ee^{-ct}
\norm{v_h},\quad v_h\in S_h,\ \alpha>0.
\end{equation}

Finally, we define the Ritz projector $R_h \colon \dot{H}^1 \to \dot{S}_h$ by
\begin{align*}
\inner{\nabla R_h v}{\nabla w_h} = \inner{\nabla v}{\nabla w_h},
\quad \forall v \in \dot{H}^1,\, w_h \in \dot{S}_h.
\end{align*}
We extend it to $R_h \colon {H}^1 \to {S}_h$ by
\begin{align}\label{R_h}
  R_hv=R_hPv+(I-P)v,  \quad v \in {H}^1.
\end{align}
We then have the following bound for $R_hv-v=(R_h-I)Pv$ (cf.\
\cite[Ch.~1]{Vidar})
\begin{equation}\label{chc2:Rherror}
\norm{R_h v -v} \le Ch^{\beta} \snorm[\beta]{v},
\quad v \in {H}^\beta, \ \beta \in [1,2].
\end{equation}
In order to simplify the presentation, we assume that $P_h$ is bounded
with respect to the $H^1$ and $L_4$ norms, and that we have an inverse
bound for $A_h$,
\begin{equation}\label{x}
\begin{aligned}
& \norm[1]{P_h v} \le C\norm[1]{v},&& v \in {H}^1,\\
& \norm[L_4]{P_h v} \le C \norm[L_4]{v}, && v \in L_4(\DD), \\
& \norm{A_h v_h} \le C h^{-2}\norm{v_h}, && v_h \in S_h.
\end{aligned}
\end{equation}
This holds, for example, if the mesh family $\lbrace \cT_h \rbrace_{h > 0}$ is quasi-uniform.
\subsection{The Wiener process}
We recall the definitions of the trace and the Hilbert-Schmidt norm of
a linear operator $T$ on $H$:
\begin{equation} \label{eq:defTr}
\Tr(T) = \sum_{k=1}^\infty \inner{T f_k}{f_k},\quad
 \norm[\HS]{T} = \Big( \sum_{k=1}^\infty \norm{Tf_k}^2 \Big)^{\frac{1}{2}},
\end{equation}
where $\lbrace f_k \rbrace_{k=1}^\infty$ is an arbitrary orthonormal basis of $H$.

Let $(\Omega,\cF,\IP)$ be a probability space.
Let $Q$ be a selfadjoint, positive semidefinite, bounded, linear
operator on $H$ with $\Tr(Q) < \infty$. Let $\lbrace e_k \rbrace_{k
  =1}^\infty$ be an orthonormal eigenbasis for $Q$ with eigenvalues
$\lbrace \gamma_k \rbrace_{k=1}^\infty$. Then we define the $Q$-Wiener
process
\begin{equation*}
W(t) = \sum_{k=1}^\infty \gamma_k^{\frac{1}{2}} \beta_k(t) e_k,
\end{equation*}
where the $\beta_k$ are real-valued, independent Brownian motions. The
series converges in $L_2(\Omega,H)$; that is, with respect to the norm
$\norm[L_2(\Omega,H)]{v}=(\EE[\norm{v}^2])^{\frac12}$. The process $W$
generates a filtration $\{\cF_t\}_{t\ge0}$ so that it becomes a square
integrable martingale and so that we can integrate with respect to
$W$. In the sequel we work in the resulting filtered probabality space
$(\Omega,\cF,\mathbf{P},\lbrace\cF_t\rbrace_{t \ge 0})$. We refer to
\cite{DaPratoZabczyk} for the details. The $Q$-Wiener process can be
defined also when the covariance operator has infinite trace but this
is not needed in the present work.
\subsection{The stochastic convolution}
We now define (cf.~\cite{DaPratoZabczyk})
\begin{equation}  \label{convA}
  \begin{split}
W_A(t)
&= \int_0^t \ee^{-(t-s)A^2} \, \dd W(s)
\\&
= \int_0^t \ee^{-(t-s)A^2} P\, \dd W(s)
+ \int_0^t \inner{\dd W(s)}{\varphi_0}\varphi_0
\\&
= \int_0^t \ee^{-(t-s)A^2} P\, \dd W(s)
+ \inner{W(t)}{\varphi_0}\varphi_0
\\&
= \int_0^t \ee^{-(t-s)A^2} P\, \dd W(s)
+ (I-P)W(t).
\end{split}
\end{equation}
Similarly,
\begin{equation} \label{convAh}
\begin{split}
W_{A_h}(t)
& = \int_0^t \ee^{-(t-s)A_h^2} P_h\, \dd W(s)\\
& = \int_0^t \ee^{-(t-s)A_h^2} P_h P\, \dd W(s)
+ \inner{W(t) } {\varphi_0}\varphi_0\\
& = \int_0^t \ee^{-(t-s)A_h^2} P_h P\, \dd W(s)
+ (I-P)W(t).
\end{split}
\end{equation}
Hence, the constant eigenmodes cancel:
\begin{equation} \label{chc2:cancel}
W_{A_h}(t) - W_A(t)
= \int_0^t \big( \ee^{-(t-s)A_h^2}P_h - \ee^{-(t-s)A^2}\big)P\,\dd W(s).
\end{equation}

These convolutions were studied in \cite{Ali}. We quote the following
results from there. We use the norms
\begin{align*}
  \norm[L_2(\Omega,\dot{H}^\beta)]{v}=\big(\EE\big[\,\snorm[\beta]{v}^2\,\big]\big)^{\frac12}.
\end{align*}
\begin{thm}\label{Thm3.1}
  If $\norm[\HS]{A^{\frac{\beta - 2}{2}} Q^{\frac{1}{2}}} < \infty$
  for some $\beta \ge 2$, then
 \begin{equation*}
  \norm[L_2(\Omega , \dot{H}^{\beta})]{W_A(t)}
\le C\norm[\HS]{A^{\frac{\beta - 2}{2}} Q^{\frac{1}{2}}}, \quad t \ge 0.
 \end{equation*}
\end{thm}
\begin{thm}\label{Thm3.2}
If  $\norm[\HS]{Q^{\frac{1}{2}}} < \infty$, then
\begin{equation*}
 \norm[L_2(\Omega , H)]{W_{A_h}(t) - W_A(t)}
   \le Ch^{2}  \abs{\log h}\norm[\HS]{Q^{\frac{1}{2}}}, \quad t \ge 0.
\end{equation*}
\end{thm}
Note that $\beta=2$ in the latter theorem.  In \cite{Ali} these are
stated with a wider range of the order $\beta$, but this is
not needed in the present work.
\subsection{Gronwall's lemma}
We need a generalization of Gronwall's
lemma. A proof can found in \cite{Stig}.
\begin{lem}[Generalized Gronwall lemma]\label{GWlem}
Let $ \varphi\in L_1([0,T],\IR) $ be a nonnegative function. If
\begin{align*}
\varphi(t) \le At^{-1+\alpha} + B \int_0^t
(t-s)^{-1+\beta}\varphi(s)\,\dd s,
\quad t \in (0,T],
\end{align*}
with constants $ A,B \ge 0 $ and $\alpha, \beta > 0$, then
there is a constant $ C = C(B,T,\alpha,\beta) $ such that
\begin{align*}
\varphi(t) \le CAt^{-1+\alpha}, \quad t \in (0,T].
\end{align*}
\end{lem}

We also use the standard Gronwall lemma:

\begin{lem}[Gronwall's lemma]\label{ourGW}
Let $ \varphi\in L_1([0,T],\IR) $. If
\begin{equation*}
\varphi(t) \le A + Ct + B \int_0^t \varphi(s)\,\dd s, \quad t  \in[0,T],
\end{equation*}
for some constants $A,C \ge 0$ and $B > 0$, then
\begin{align*}
\varphi(t) \le \Big(A+\frac{C}{B}\Big) \ee^{Bt},\quad t \in[0,T].
\end{align*}
\end{lem}

\subsection{Bounds for the nonlinear term}
Recall that the standard Sobolev norm $\norm[H^k]{\cdot}$ is
equivalent to the norm $\norm[k]{\cdot}$ in \eqref{eq:nondotnorm} for
integer $k\ge0$.
\begin{lem}\label{lemma9.2}
For $ u,v \in H^3 $ and $f(s)=s^3-s$ we have
\begin{align}\label{Deltaf}
\norm{\Delta f(u)} &\le C \big(1+\norm[1]{u}^2\big) \norm[3]{u},\\
\label{Deltauv-1}
\norm{A_h^{-\frac{1}{2}}P\big(f(u) - f(v)\big)}
&\le C\big( 1+\norm[1]{u}^2 + \norm[1]{v}^2 \big) \norm{u-v}.
\end{align}
\end{lem}

\begin{proof}  We have $f'(s)=3s^2-s$, $f''(s)=6s$.
  Using H\"older's inequality, Sobolev's inequality $\norm[L_6]{u} \le
  C\norm[H^1]{u}$ (for $d\le3$), and $\norm[H^k]{u} \le C\norm[k]{u}$,
  we get
\begin{align*}
  \begin{split}
\norm{\Delta f(u)}
&=\norm{f'(u)\Delta u+f''(u)\abs{\nabla u}^2}
\\&
\le \norm[L_3]{f'(u)}\norm[L_6]{\Delta u}
+   \norm[L_6]{f''(u)}\norm[L_6]{\nabla u}^2
\\&
\le C \big(1+\norm[L_6]{u}^2\big)\norm[L_6]{\Delta u}
+ C \norm[L_6]{u}\norm[L_6]{\nabla u}^2
\\&
\le C \big(1+\norm[H^1]{u}^2\big)\norm[H^3]{u}
+ C \norm[H^1]{u}\norm[H^2]{u}^2
\\&
\le C \big(1+\norm[1]{u}^2\big)\norm[3]{u}
+C \norm[1]{u}\norm[2]{u}^2
\\&
\le C \big(1+\norm[1]{u}^2\big)\norm[3]{u},
\end{split}
\end{align*}
where we used
$\norm[2]{u}\le C\norm[1]{u}^{\frac{1}{2}}\norm[3]{u}^{\frac{1}{2}}$
in the last step.
This proves \eqref{Deltaf}.

For \eqref{Deltauv-1} we apply \eqref{chc2:a1} and the H\"older and
Sobolev inequalities {($d\le 3$)} to get
\begin{align*}
  \begin{split}
\norm{A_h^{-\frac{1}{2}} P \varphi}
&=\sup_{v_h\in{S}_h}\frac{\inner{A_h^{-\frac{1}{2}} P \varphi}{v_h}}{\norm{v_h}}
=\sup_{v_h\in{S}_h}\frac{\inner{ \varphi}{A_h^{-\frac{1}{2}} P v_h}}{\norm{v_h}}
\\&
=\sup_{w_h\in\dot{S}_h}\frac{\inner{\varphi}{w_h}}{|w_h|_1}
\le \sup_{w_h\in\dot{S}_h}\frac{\norm[L_{6/5}]{\varphi}\norm[L_6]{w_h}}{|w_h|_1}
\le C \norm[L_{{6}/{5}}]{\varphi}.
\end{split}
\end{align*}
We use this with $\varphi = f(u)-f(v)= \int_0^1f'(u_s)\,\dd s\,(u-v)$,
where $u_s = su +(1-s)v$, and H\"older's and Sobolev's inequalities to
get
\begin{align*}
&\norm[]{A_h^{-\frac{1}{2}}P\big(f(u) - f(v)\big)}
=
\norm{A_h^{-\frac{1}{2}} P {\varphi} }
\le C \norm[L_{{6}/{5}}]{{\varphi}}\\
&\qquad \le C \int_0^1\norm[L_3]{f'(u_s)}\,\dd s\, \norm{u-v}
 \le C\int_0^1\big(1+\norm[L_6]{u_s}^2\big)\,\dd s\,\norm{u - v}\\
&\qquad \le C\int_0^1\big(1+\norm[1]{u_s}^2\big)\,\dd s\, \norm{u-v}
 \le C\big(1+\norm[1]{u}^2 + \norm[1]{v}^2\big)\norm{u-v}.
\end{align*}
This is \eqref{Deltauv-1}.
\end{proof}
\section{The Cahn-Hilliard-Cook equation}
\subsection{The continuous problem}
The Cahn-Hilliard-Cook equation is
\begin{equation}\label{C-H-C}
\begin{aligned}
& \dd u - \Delta w \,\dd t = \dd W && \text{in } \ \cD\times (0,T],\\
& w = -\Delta u +f(u)   && \text{in } \ \cD\times (0,T],\\
& \frac{\pt u}{\pt n} = \frac{\pt w}{\pt n} = 0
&& \text{on } \ \pt\cD\times (0,T],\\
& u(0) = u_0  && \text{in } \ \cD.
\end{aligned}
\end{equation}
The finite element approximation is based on its weak form, which is
(formally)
\begin{equation}\label{wCHC}
\begin{aligned}
& \inner{u(t)}{v} - \inner{u_0}{v}
 + \int_0^t \inner{\nabla w(s)}{\nabla v}\,\dd s
 = \int_0^t \inner{\dd W(s)}{v},&&t\in (0,T],\\
& \inner{w}{v}=\inner{\nabla u}{ \nabla v}+\inner{f(u)}{v},&& t\in (0,T],
\end{aligned}
\end{equation}
for all $v \in \dot{H}^1$. With
the operator $A$, defined in \subsec\ref{norms}, we
write \eqref{C-H-C} in the formal abstract form on
$H=L_2(\mathcal{D})$:
\begin{equation}\label{C-H-C-X}
\dd X + \big( A^2 X + Af(X) \big)\,\dd t = \dd W,\quad t\in (0,T];
\quad X(0) = X_0.
\end{equation}
A \textit{weak solution} of \eqref{C-H-C-X} is an adapted $H$-valued
process $X$, which is continuous almost surely and satisfies the equation
\begin{align}\label{weaksol}
\inner{X(t)}{v} - \inner{X_0}{v}
+ \int_0^t \big(\inner{X(s)}{A^2 v}
+ \inner{f(X(s))}{A v}\big)\,\dd s
= \int_0^t \inner{\dd W(s)}{v}
\end{align}
almost surely for all $v \in \dot{H}^4=D(A^2)$, $t\in [0,T]$, where we
also require the integrand in the deterministic integral to be in
$L_1([0,T],\IR)$ almost surely.  A \textit{mild solution} of
\eqref{C-H-C-X} is an adapted $H$-valued process $X$, continuous
almost surely, which satisfies
\begin{align}\label{mildC-H-C-X}
X(t) =  \ee^{-tA^2}X_0 - \int_0^t A\ee^{-(t-s)A^2}f(X(s))\,\dd s
+ \int_0^t \ee^{-(t-s)A^2}\,\dd W(s),
\end{align}
almost surely for $t\in [0,T]$, where we also require that the first
integrand is in $L_1([0,T],H)$ and the stochastic integral exists
almost surely.
\subsection{The finite element problem}
Recalling \eqref{wCHC}, we define the finite element solution $u_h(t),
w_h(t) \in S_h$ of \eqref{C-H-C} by
\begin{align*}
\begin{aligned}
& \inner{u_h(t)}{v_h}
 - \inner{u_0}{v_h}+ \int_0^t \inner{\nabla w_h(s)}{\nabla v_h}\,\dd s
= \int_0^t \inner{\dd W(s)}{v_h},\quad &&t\in(0,T],\\
& \inner{w_h}{v_h}=\inner{\nabla u_h}{\nabla v_h}
  +\inner{f(u_h)}{v_h},
&& t\in(0,T],
\end{aligned}
\end{align*}
for all $v_h \in S_h$. With the operators $A_h$,
$P_h$ from \subsec \ref{subsec:fem} we write this as an abstract
equation in $S_h$:
\begin{equation}\label{FEC-H-C-X}
\dd X_h + \big( A_h^2 X_h + A_h P_h f(X_h) \big)\,\dd t
= P_h \, \dd W,\quad t\in(0,T];\quad X_h(0) = P_h X_0.
\end{equation}
Since $S_h$ is finite-dimensional and $f$ is a polynomial, it is easy
to see using standard arguments that \eqref{FEC-H-C-X} has a unique
solution $X_h$, adapted, continuous almost surely,
satisfying both
$$
X_h(t) - P_hX_0
+ \int_0^t \big(A_h^2X_h(s)
+  A_hP_hf(X_h(s))\big)\,\dd s = \int_0^tP_h\,\dd W(s),
$$
and
\begin{align*}\label{mildFEC-H-C-X}
X_h(t) =  \ee^{-tA_h^2}P_hX_0 - \int_0^t \ee^{-(t-s)A_h^2}A_h P_h
f(X_h(s))\,\dd s
+ \int_0^t \ee^{-(t-s)A_h^2}P_h\,\dd W(s) ,
\end{align*}
almost surely for $t\in [0,T]$.
\subsection{A Lyapunov functional}
Define the functional
\begin{equation}\label{lyapunov}
J(u) = \frac{1}{2} \norm{\nabla u}^2 + \int_{\cD} F(u) \, \dd x, \quad u \in H^1,
\end{equation}
where $F(s) = \frac{1}{4}s^4 - \frac{1}{2}s^2$ is a primitive of $f(s)
= s^3 - s$. This is a Lyapunov functional for the deterministic
Cahn-Hilliard equation, which means that in the deterministic case
$J(X(t))$ does not increase along solution paths. For the stochastic
equation this is not true, but we have a bound for the expected value
of $J(X(t))$.

\begin{thm}\label{lyap-thm}
  Assume that $\norm[\HS]{A^{\frac{1}{2}}Q^{\frac{1}{2}}} < \infty$ and that $X_0$ is
  $\cF_0$-measurable with values in $H^1$ satisfying $\IE[J(X_0)] < \infty$. If
  $X$ is a weak solution of \eqref{C-H-C-X} and $X_h$ is the solution of
  \eqref{FEC-H-C-X}, then, for all $t>0$, we have
\begin{equation}\label{JX}
\IE[J(X(t))]+\IE\Big[\int_0^t \snorm[1]{J'(X(s))}^2 \,\dd s \Big]
\le C \Big( \IE[J(X_0)] + K_Qt  + K_Q ^2t^{2}   \Big)
\end{equation}
and
\begin{equation}\label{JXh}
\IE[J(X_h(t))]+\IE\Big[\int_0^t \snorm[1]{J'(X_h(s))}^2 \,\dd s \Big]
\le C \Big( \IE[J(P_h X_0)] +  K_Qt  + K_Q ^2 t^{2}  \Big),
\end{equation}
where $K_Q = \norm[\HS]{A^{\frac{1}{2}}Q^{\frac{1}{2}}}^2
+\norm[\HS]{Q^{\frac{1}{2}}}^2.$
\end{thm}

\begin{proof}
  We prove \eqref{JXh}; the proof of \eqref{JX} is obtained in
    a similar way by approximating \eqref{C-H-C-X} by Galerkin's
    method based on the eigenbasis of $A$ instead of the finite
    element Galerkin method used in \eqref{FEC-H-C-X} (see also
    \cite{Prato}).

  We consider \eqref{FEC-H-C-X} as an It\^o differential equation in
  $S_h$ driven by $P_hW$, which is a $Q_h$-Wiener process in $S_h$
  with $Q_h=P_hQP_h$. By assumption \eqref{x} it follows that $
  \IE[J(P_h X_0)]<\infty$, if $ \IE[J(X_0)]<\infty$.  By applying
  It\^o's formula (\cite[Theorem 4.17]{DaPratoZabczyk}) to
  $J(X_h(t))$, we obtain
\begin{align*}
J(X_h(t))
& = J(X_h(0)) + \int_0^t \inner{J'(X_h(s))}{\dd X_h(s)}
+ \frac{1}{2}\int_0^t \Tr(J''(X_h(s)Q_h))\,\dd s \\
& =  J(P_hX_0) +
\int_0^t \inner{J'(X_h(s))}{-A_h^2 X_h(s) - A_h P_h f(X_h(s))}\,\dd s\\
& \quad + \int_0^t \inner{J'(X_h(s))}{P_h\,\dd W(s)}
+ \frac{1}{2}\int_0^t \Tr(J''(X_h(s)Q_h))\,\dd s .
\end{align*}
With a slight abuse of notation we consider here $J$ as a function
$S_h\to\IR$ and we compute $J'(u_h)\in S_h$ and $J''(u_h)\colon S_h\to
S_h$ as follows:
\begin{align*}
\inner{J'(u_h)}{v_h}
= \inner{\nabla u_h}{\nabla v_h} + \inner{f(u_h)}{v_h}
= \inner{A_h u_h + P_h f(u_h)}{v_h}
\end{align*}
and
\begin{align*}
\inner{J''(u_h)v_h}{w_h}
 = \inner{\nabla v_h}{\nabla w_h} + \inner{f'(u_h)v_h}{w_h}
 = \inner{A_h v_h + P_h [f'(u_h)v_h]}{w_h}
\end{align*}
for $u_h,v_h,w_h\in S_h$, so that
\begin{align} \label{eq:JP}
J'(u_h) = A_h u_h + P_h f(u_h),  \quad
J''(u_h) = A_h + P_h[f'(u_h)\,\cdot\,].
\end{align}
Hence, by \eqref{chc2:a1},
\begin{align}\label{J(X_h)}
  \begin{split}
J(X_h(t)) +\int_0^t \snorm[1]{J'(X_h(s))}^2 \,\dd s
& =  J(P_hX_0) + \int_0^t \inner{J'(X_h(s))}{P_h\,\dd W(s)}\\
& \quad
+ \frac{1}{2}\int_0^t \Tr(J''(X_h(s)Q_h))\,\dd s .
\end{split}
\end{align}
The stochastic integral is a martingale, so that
$\IE[\int_0^t \inner{J'(X_h)}{P_h\,\dd W}]=0$, and hence
\begin{equation}\label{EJmain}
  \begin{split}
&\IE[J(X_h(t))] +\IE\Big[\int_0^t \snorm[1]{J'(X_h(s))}^2 \,\dd s  \Big]
\\&\qquad
=  \IE[J(P_hX_0)]
+ \frac{1}{2} \IE \Big[ \int_0^t \Tr(J''(X_h(s))Q_h)  \,\dd s \Big].
\end{split}
\end{equation}
We now compute
\begin{align*}
  \Tr(J''(X_h(s))Q_h)
   = \Tr(A_hQ_h)+\Tr(P_h[f'(X_h(s))\,\cdot\,]Q_h)
\end{align*}
by the definition in \eqref{eq:defTr}. To this end let $\lbrace
\varphi_{h,j}\rbrace_{j=0}^{N_h}$ be an orthonormal basis of
eigenvectors of $A_h$ and $\lbrace \lambda_{h,j}\rbrace_{j=0}^{N_h}$
the corresponding eigenvalues.  Then, since
$A_h\varphi_{h,0}=0$ and $\norm[\HS]{T}=\norm[\HS]{T^*}$,
\begin{align*}
\Tr(A_hQ_h)
&= \sum_{j=1}^{N_h} \inner{A_h Q_h \varphi_{h,j}}{\varphi_{h,j}}
 = \sum_{j=1}^{N_h} \inner{P_h Q P_h \varphi_{h,j}}{A_h \varphi_{h,j}}
\\&
= \sum_{j=1}^{N_h} \lambda_{h,j} \inner{Q\varphi_{h,j}}{\varphi_{h,j}}
= \sum_{j=1}^{N_h} \inner{Q^{\frac{1}{2}}A_{h}^{\frac{1}{2}}\varphi_{h,j}}{Q^{\frac{1}{2}}A_{h}^{\frac{1}{2}}\varphi_{h,j}}
\\&
= \sum_{j=1}^{N_h} \norm{Q^{\frac{1}{2}}A_h^{\frac{1}{2}}P_h\varphi_{h,j}}^2
=\norm[\HS]{Q^{\frac{1}{2}} A_h^{\frac{1}{2}}P_h}^2
=\norm[\HS]{(Q^{\frac{1}{2}} A_h^{\frac{1}{2}}P_h)^*}^2
\\&
= \norm[\HS]{A_h^{\frac{1}{2}} P_h Q^{\frac{1}{2}}}^2
= \norm[\HS]{A_h^{\frac{1}{2}} P_h A^{-\frac{1}{2}} A^{\frac{1}{2}} Q^{\frac{1}{2}}}^2
\\&
\le \norm[B(\dot{H})]{A_h^{\frac{1}{2}} P_h A^{-\frac{1}{2}}}^2
\norm[\HS]{A^{\frac{1}{2}} Q^{\frac{1}{2}}}^2.
\end{align*}
Here we use \eqref{chc2:a1} and \eqref{x} to get
\begin{align*}
\norm{A_h^{\frac{1}{2}} P_h A^{-\frac{1}{2}}v}
= \snorm[1]{P_h A^{-\frac{1}{2}} v}
\le C \snorm[1]{A^{-\frac{1}{2}}v}
= C\norm{v},\quad v \in \dot{H},
\end{align*}
so that $\norm[B(\dot{H})]{A_h^{\frac{1}{2}}P_hA^{-\frac{1}{2}}}
\le C$. Hence, with $K_Q =
\norm[\HS]{A^{\frac{1}{2}}Q^{\frac{1}{2}}}^2 +\norm[\HS]{Q^{\frac{1}{2}}}^2$,
\begin{equation}\label{AQ}
\Tr(A_h Q_h)
\le  C \norm[\HS]{A^{\frac{1}{2}}Q^{\frac{1}{2}}}^2
\le CK_Q.
\end{equation}

Let $ \lbrace e_{h,k}\rbrace_{k=0}^{N_h} $ be an orthonormal
eigenbasis of $ Q_h $ and $ \lbrace\gamma_{h,k}\rbrace_{k=0}^{N_h} $
the corresponding eigenvalues. We get
\begin{align}  \label{xxx}
  \begin{split}
&\Tr\big(P_h[f'(X_h)\,\cdot\,]Q_h \big)
 = \sum_{k=0}^{N_h} \inner{P_h[f'(X_h) Q_h e_{h,k}]}{e_{h,k}}\\
&\qquad = \sum_{k=0}^{N_h} \gamma_{h,k} \inner{f'(X_h)e_{h,k}}{e_{h,k}}
= \sum_{k=0}^{N_h} \inner{f'(X_h) Q_h^{\frac{1}{2}} e_{h,k}}{Q_h^{\frac{1}{2}} e_{h,k}}.
\end{split}
\end{align}
By using the bound $\abs{f'(s)} \le C(1+s^2)$ and H\"older's
and Sobolev's inequalities we get
\[
\abs{\inner{f'(u)v}{v}}
\le C(1+\norm[L_4]{u}^2) \norm[L_4]{v}^2
\le C(1+\norm[L_4]{u}^2) \norm[H^1]{v}^2
\le C(1+\norm[L_4]{u}^2) \norm[1]{v}^2.
\]
By \eqref{xx} and \eqref{chc2:a1} we have, for $v_h
\in S_h$,
\begin{align*}
\norm[1]{v_h}^2
= \snorm[1]{v_h}^2 + \inner{v_h}{\varphi_0}^2
= \norm{A_h^{\frac{1}{2}}v_h}^2 + \inner{v_h}{\varphi_0}^2,
\end{align*}
so that, by \eqref{AQ},
\begin{align*}
&\sum_{k=0}^{N_h} \norm[1]{Q_h^{\frac{1}{2}}e_{h,k}}^2
= \sum_{k=0}^{N_h} \norm{A_h^{\frac{1}{2}}Q_h^{\frac{1}{2}}e_{h,k}}^2
+ \sum_{k=0}^{N_h} \inner{Q_h^{\frac{1}{2}}e_{h,k}}{\varphi_0}^2
\\& \qquad
= \sum_{k=0}^{N_h} \gamma_{h,k}\inner{A_he_{h,k}}{e_{h,k}}
+ \sum_{k=0}^{N_h} \gamma_{h,k}\inner{e_{h,k}}{\varphi_0}^2
\\& \qquad
\le  \Tr(A_hQ_h)+\Tr(Q_h)
\le  \Tr(A_hQ_h)+\Tr(Q)
\\& \qquad
\le C \norm[\HS]{A^{\frac{1}{2}}Q^{\frac{1}{2}}}^2+ \norm[\HS]{Q^\frac12}^2
\le CK_Q.
\end{align*}
Returning to \eqref{xxx}, we now have
\begin{equation*} \label{Phdot}
|\Tr( P_h [f'(X_h)\,\cdot\,]Q_h)|
\le C\big(1+\norm[L_4]{X_h}^2\big) \sum_{k=0}^{N_h} \norm[1]{Q_h^{\frac{1}{2}}e_{h,k}}^2
\le C\big(1+\norm[L_4]{X_h}^2\big)K_Q,
\end{equation*}
Using also \eqref{AQ} we conclude
\begin{align} \label{eq:foo}
|\Tr(J''(X_h)Q_h)|\le C K_Q\big(1+\norm[L_4]{X_h}^2\big).
\end{align}
It remains to relate $\norm[L_4]{X_h}$ to
$J(X_h)$.  By definition of the Lyapunov functional \eqref{lyapunov}
and noting that $F(s) = \frac{1}{4} s^4 - \frac{1}{2} s^2 \ge c_1 s^4
- c_2$, we get (with new constants)
\begin{align*}
J(u) \ge \frac{1}{2}\norm{\nabla u}^2 + C_1 \norm[L_4]{u}^4 - C_2,
\end{align*}
which implies
\begin{align}  \label{eq:jx}
\norm{\nabla u}^2+\norm[L_4]{u}^4 \le C_3\big(1+J(u)\big).
\end{align}
Hence, in view of \eqref{eq:foo},
\begin{align}
\label{eq:TrJQ}
|\Tr(J''(X_h)Q_h)|\le C K_Q\big(1+J(X_h)^{\frac12}\big).
\end{align}
Inserting this into \eqref{EJmain} gives
\begin{equation}\label{EJXhfinal}
  \begin{split}
&\IE[J(X_h(t))]+\IE\Big[\int_0^t \snorm[1]{J'(X_h(s))}^2 \,\dd s \Big]
\\&\qquad
\le  \IE[J(P_h X_0)]
+ C K_Q  \Big( t + \int_0^t \IE\big[ J(X_h(s))^{\frac12}\big]\,\dd s \Big).
\end{split}
\end{equation}
Here, by H\"older's and Young's inequalities, we have, for $\epsilon > 0$,
\begin{align*}
\begin{split}
C K_Q  \int_0^t \IE[J(X_h(s))^{\frac12}]\,\dd s
& \le C K_Q t^{\frac{1}{2}} \Big( \int_0^t \IE[J(X_h(s))]\,\dd s \Big)^{\frac{1}{2}}
\\&
\le \epsilon  \int_0^t \IE[J(X_h(s))]\,\dd s + C\epsilon^{-1} tK_Q^2.
\end{split}
\end{align*}
Putting this in \eqref{EJXhfinal} gives
\begin{align*}
&\IE[J(X_h(t))]+\IE\Big[\int_0^t \snorm[1]{J'(X_h(s))}^2 \,\dd s \Big]
\\&\qquad
\le \IE[J(P_hX_0)]
+ C\big( K_Q  + \epsilon^{-1}K_Q ^2 \big) t
+ \epsilon\int_0^t \IE[J(X_h(s))]\,\dd s.
\end{align*}
We apply the Gronwall Lemma \ref{ourGW} to get,
\begin{align*}
&\IE[J(X_h(t))]+\IE\Big[\int_0^t \snorm[1]{J'(X_h(s))}^2 \,\dd s \Big]
\\&\qquad
\le \ee^{\epsilon t}\Big( \IE[J(P_hX_0)]
+ C( \epsilon^{-1} K_Q  + \epsilon^{-2} K_Q ^2)  \Big)\\
\\&\qquad
\le \ee \Big( \IE[J(P_hX_0)] + C( t K_Q  + t^2 K_Q ^2)  \Big),
\end{align*}
where for each fixed $t$ we have chosen $\epsilon = t^{-1}$ to get an optimal bound.
\end{proof}

This theorem is adapted from \cite{Prato}. We have improved it in
several ways. First the growth of the bound is reduced from
exponential to quadratic with respect to $t$.  Most importantly, we
have removed the assumption that $A$ and $Q$ have a common eigenbasis
and that the eigenbasis satisfies $\norm[L_\infty]{e_k}\le C$ for all
$k$.  This is important because even if $A$ and $Q$ commute, this will
not be true for $A_h$ and $Q_h$.  This is crucial for the proof of the
bound for $X_h$.

Note that $\norm[\HS]{A^{\frac12}Q^{\frac12}}<\infty$ implies $K_Q =
  \norm[\HS]{A^{\frac{1}{2}}Q^{\frac{1}{2}}}^2
  +\norm[\HS]{Q^{\frac{1}{2}}}^2<\infty$.  This is because of the
  boundedness of $A^{-\frac12}$ and
\begin{align} \label{KQ}
  \begin{split}
 \norm[\HS]{Q^{\frac{1}{2}}}^2
  &=\sum_{j=1}^\infty\inner{Q\varphi_j}{\varphi_j}+\inner{Q\varphi_0}{\varphi_0}
  =\sum_{j=1}^\infty\norm{PQ^\frac12\varphi_j}^2+\inner{Q\varphi_0}{\varphi_0} \\
  &=\sum_{j=1}^\infty\norm{A^{-\frac12}A^{\frac12}PQ^\frac12\varphi_j}^2+\inner{Q\varphi_0}{\varphi_0} \\
  &\le C\sum_{j=1}^\infty\norm{A^{\frac12}PQ^\frac12\varphi_j}^2+\inner{Q\varphi_0}{\varphi_0} \\
  &\le C\norm[\HS]{A^{\frac12}PQ^\frac12}^2+\inner{Q\varphi_0}{\varphi_0}
  = C\norm[\HS]{A^{\frac12}Q^\frac12}^2+\inner{Q\varphi_0}{\varphi_0}<\infty .
\end{split}
\end{align}
This condition is therefore the same as the condition for regularity
of order $\beta=3$ for $W_A(t)$ in Theorem \ref{Thm3.1}.

We now use the previous theorem together with Chebyshev's inequality
to obtain pathwise norm bounds uniformly on subsets of $\Omega$ with
probability arbitrarily close to $1$.  In order to achieve this we
first replace the bound of $\sup_{s\in[0,t]}\IE[J(X(s))]$ from
Theorem \ref{lyap-thm} by a bound for
$\IE[\sup_{s\in[0,t]}(\norm{\nabla X(s)}^2 +
\norm[L_4]{X(s)}^4)]$.

\begin{coro}\label{coro2.5}
  Assume that $\norm[\HS]{A^{\frac12}Q^{\frac12}}<\infty$ and that
  $X_0$ is $\cF_0$-measurable with values in $H^1$ satisfying
  \begin{align} \label{eq:cororho}
  \norm[L_2(\Omega,H^1)]{X_0}^2 + \norm[L_4(\Omega,L_4)]{X_0}^4 \le \rho.
  \end{align}
   If $X$ is a weak solution of \eqref{C-H-C-X} and $X_h$ is the solution of
  \eqref{FEC-H-C-X},
  then, for $T\ge0$,
  \begin{align}
    \label{eq:2}
\IE\Big[\sup_{t\in[0,T]}
\big(\norm{\nabla X(t)}^2 + \norm[L_4]{X(t)}^4\big) \Big]
&\le K_T, \\
    \label{eq:3}
\IE\Big[\sup_{t\in[0,T]}
\big(\norm{\nabla X_h(t)}^2 + \norm[L_4]{X_h(t)}^4\big) \Big]
&\le K_T,
 \end{align}
where $K_T$ depends on $\rho,K_Q, T$.
Moreover, for every $\epsilon \in
  (0,1) $, there is $\Omega_\epsilon \subset \Omega$ with
  $\IP(\Omega_\epsilon) \ge 1-\epsilon$ and
\begin{align}
 \norm{\nabla X(t)}^2 + \norm[L_4]{X(t)}^4
 &\le \epsilon^{-1}K_T \quad \text{on} \ \Omega_\epsilon, \ t \in [0,T],\label{coroone}\\
 \norm{\nabla X_h(t)}^2 + \norm[L_4]{X_h(t)}^4
 &\le \epsilon^{-1}K_T \quad \text{on} \ \Omega_\epsilon, \ t \in [0,T],\label{corotwo}\\
 \norm[1]{X(t)}^2 + \norm[1]{X_h(t)}^2
 &\le \epsilon^{-1} K_T \quad \text{on} \ \Omega_\epsilon, \ t \in [0,T],\label{corothree}\\
 \norm[3]{W_A(t)}^2
 &\le \epsilon^{-1}K_T \quad \text{on} \ \Omega_\epsilon, \ t \in [0,T].\label{corofour}
\end{align}
\end{coro}

\begin{proof} We prove the bounds for $X_h$ and $W_A$; the others are
  proved similarly.

From \eqref{eq:cororho} and \eqref{x} there
  follows $\IE[J(P_hX_0)]\le C(1+\rho)$. Using also \eqref{eq:jx} in
  \eqref{J(X_h)}, we obtain
\begin{align*}
   &\IE\Big[\sup_{t\in[0,T]}
   \big(\norm{\nabla X_h(t)}^2 + \norm[L_4]{X_h(t)}^4\big) \Big]
   \\&\qquad
 \le C\big(1+\rho\big)
   +C \,\IE\Big[\sup_{t\in[0,T]}\Abs{
   \int_0^t \inner{J'(X_h(s))}{P_h\,\dd W(s)}}\Big]
   \\&\qquad\quad
   + C\, \IE\Big[\sup_{t\in[0,T]}\Abs{\int_0^t \Tr(J''(X_h(s)Q_h))\,\dd s }\Big].
\end{align*}
The stochastic integral is
$\int_0^t \inner{J'(X_h(s))}{P_h\,\dd W(s)}
=\int_0^t \widetilde{J'(X_h(s))}{P_h\,\dd W(s)}$,
where $\widetilde{J'(X_h(s))}\colon H\to\IR$ is defined by
$\widetilde{J'(X_h(s))}v=\inner{J'(X_h(s))}{v}$.
This integral is a martingale. Hence, we may use H\"older's
inequality, the martingale inequality (\cite[Theorem
3.8]{DaPratoZabczyk}), and the It\^o isometry (\cite[Corollary
4.14]{DaPratoZabczyk}) to get
\begin{align*}
  \begin{split}
&\Big(\IE\Big[\sup_{t\in[0,T]}\Abs{
 \int_0^t \widetilde{J'(X_h(s))}{P_h\,\dd W(s)}}\Big]\Big)^2
\le
 \IE\Big[\sup_{t\in[0,T]}\Abs{
 \int_0^t \widetilde{J'(X_h(s))}{P_h\,\dd W(s)}}^2\Big]
\\& \quad
\le
4 \sup_{t\in[0,T]}\IE\Big[
\Abs{
 \int_0^t \widetilde{J'(X_h(s))}P_h\,\dd W(s)}^2\Big]
=
4 \sup_{t\in[0,T]}\IE\Big[
\int_0^t \norm[\HS]{\widetilde{J'(X_h(s))}Q_h^{\frac12}}^2\,\dd s\Big]
\\& \quad
=
4 \IE\Big[
\int_0^T \norm[\HS]{\widetilde{J'(X_h(s))}Q_h^{\frac12}}^2\,\dd s\Big].
\end{split}
\end{align*}
Here, by \eqref{eq:nondotnorm},
\begin{align*}
  \norm[\HS]{\widetilde{J'(X_h(s))}Q_h^{\frac12}}^2
  &\le
 \norm{J'(X_h(s))}^2\norm[\HS]{Q_h^{\frac12}}^2
 \\&
\le
\big(\snorm[0]{J'(X_h(s))}^2+\inner{J'(X_h(s))}{\varphi_0}^2 \big)
\Tr(Q)
 \\&
  \le
C\big(1+\snorm[1]{J'(X_h(s))}^2+J(X_h(s)) \big)
\Tr(Q),
\end{align*}
where we used \eqref{eq:JP} and the (rough) bounds
$\snorm[0]{u}\le\snorm[1]{u}$ and
\begin{align*}
  &\inner{J'(u_h)}{\varphi_0}^2
 =|\cD|^{-1}\Big(\int_\cD\big(A_hu_h+P_h(u_h^3-u_h)\big)\,\dd x \Big)^2
  \\&\qquad
=|\cD|^{-1}\Big(\int_\cD P_h(u_h^3-u_h)\,\dd x \Big)^2
\le C\big(1+\norm[L_4]{u_h}^4\big)
\le C\big(1+J(u_h)\big).
\end{align*}
By using \eqref{JXh}, we conclude that
\begin{align*}
&\Big(\IE\Big[\sup_{t\in[0,T]}\Abs{
 \int_0^t \widetilde{J'(X_h(s))}{P_h\,\dd W(s)}}\Big]\Big)^2
\\&\qquad
\le C \Tr(Q)\Big( T +\IE\Big[\int_0^T \snorm[1]{J'(X_h(s))}^2 \,\dd s \Big]
+T\sup_{s\in[0,T]}\IE\big[J(X_h(s))\big] \Big)
\le K_T^2.
\end{align*}
Next, using \eqref{eq:TrJQ} and H\"older's inequality, we have
\begin{align*}
&\IE\Big[\sup_{t\in[0,T]}
\Abs{\int_0^t \Tr(J''(X_h(s)Q_h))\,\dd s }\Big]
\le
\IE\Big[\int_0^T |\Tr(J''(X_h(s)Q_h))|\,\dd s \Big]
\\&\qquad
\le
CK_Q\IE\Big[\int_0^T \big(1+J(X_h(s))^{\frac12}\big)\,\dd s \Big]
\\&\qquad
\le
CK_QT \big(1+\sup_{s\in[0,T]}\IE[J(X_h(s))] \big)
\le K_T,
\end{align*}
which finishes the proof of \eqref{eq:3}.

In order to prove \eqref{corotwo} we denote
\begin{align*}
F_h=\sup_{t\in[0,T]}\big(\norm{\nabla X_h(t)}^2+\norm[L_4]{X_h(t)}^4\big).
\end{align*}
We apply Chebyshev's inequality and \eqref{eq:3} to get, for every $\alpha > 0$,
\begin{equation*}
\begin{split}
&\IP\Big(\big\lbrace \omega \in \Omega:  F_h > \alpha \big\rbrace \Big)
\le \frac{1}{\alpha} \IE\big[F_h\big]
\le \frac{K_T}{\alpha} .
\end{split}
\end{equation*}
We choose $\alpha = \epsilon^{-1}K_T$ and set $\Omega_\epsilon=\big
\lbrace \omega \in \Omega : F_h \le \epsilon^{-1}K_T \big\rbrace$.
Then
\begin{equation*}
\IP(\Omega_\epsilon)
= 1- \IP\Big(\big\lbrace \omega \in \Omega: F_h > \alpha \big\rbrace \Big)
\ge 1-\epsilon
\end{equation*}
and \eqref{corotwo} holds.
For \eqref{corothree} we
note that
\begin{align*}
\norm[1]{u}^2
\le \norm{\nabla u}^2 + \norm{u}^2
\le \norm{\nabla u}^2 + C\big(1+\norm[L_4]{u}^4\big)
\end{align*}
and hence \eqref{corothree} follows from \eqref{eq:2} and \eqref{eq:3}
after an adjustment of $K_T$.  Finally, \eqref{corofour}
follows in a similar way from Theorem \ref{Thm3.1} with $\beta =3$
with a constant which can be absorbed in $K_T$.
\end{proof}

\section{Regularity of the solution}
We quote the following from \cite{Prato}.  There it is assumed that
$A$ and $Q$ commute and that the eigenfunctions of $A$ are uniformly
bounded in the sup norm but it can be verified that these are not
necessary for the following result.  Recall the definitions of weak
and mild solutions in \eqref{weaksol} and \eqref{mildC-H-C-X}.

\begin{thm} \label{pratothm} Let $T>0$ and assume that
  $\Tr(A^{\delta-1}Q)<\infty$ for some $\delta>0$ and that $X_0$ is
  $\cF_0$-measurable with values in $H$.  Then there is a unique weak
  solution $X$ of \eqref{C-H-C-X}.
\end{thm}

\begin{coro}\label{mildxy}
  Assume that $\norm[\HS]{A^{\frac12}Q^{\frac12}}<\infty$ and that
  $X_0$ is $\cF_0$-measurable with values in $H^1$ satisfying
  $\norm[L_2(\Omega,H^1)]{X_0}^2 + \norm[L_4(\Omega,L_4)]{X_0}^4 \le
  \rho$ for some $\rho$. Then the weak solution $X$ of
  \eqref{C-H-C-X} is also a mild solution.
\end{coro}

\begin{proof}
  By \eqref{KQ}, the condition
  $\norm[\HS]{A^{\frac12}Q^{\frac12}}<\infty$ implies
  $\Tr(A^{\delta-1}Q)<\infty$ with $\delta=1$ and hence there is a
  unique weak solution $X$ of \eqref{C-H-C-X} by Theorem
  \ref{pratothm}.  Let $\epsilon>0$ and $\Omega_\epsilon$ the set
  defined in Corollary \ref{coro2.5}. We first show that $X$ satisfies
  \eqref{mildC-H-C-X} on $\Omega_\epsilon$.  By the uniqueness weak
  solutions of \eqref{C-H-C-X}, we only have to show that the
  right-hand side of \eqref{mildC-H-C-X} satisfies \eqref{weaksol} on
  $\Omega_\epsilon$. Since $W_A$ is the unique weak solution of
  $\dd Z+A^2Z\,\dd t=\dd W,~Z(0)=0$ (by \cite[Theorem
  5.4]{DaPratoZabczyk}), it is enough to show that
$$
  Y(t):=\ee^{-tA^2}X_0
   - \int_0^t A\ee^{-(t-s)A^2}f(X(s))\,\dd s
$$
satisfies
$$
\inner{Y(t)}{v} - \inner{X_0}{v}
+ \int_0^t \big(\inner{Y(s)}{A^2 v}
+ \inner{f(X(s))}{A v}\big)\,\dd s
=0\text{ on }\Omega_\epsilon,~v\in D(A^2).
$$
This follows by standard arguments (see, e.g., \cite{Ball}) as, by
Sobolev's inequality and Corollary \ref{coro2.5}, there is $C$
depending on $\epsilon$, $\rho$, $Q$, and $T$ such that
$$
\|f(X(s)\|\le C (\|X(s)\|_{L^2}+\|X(s)\|_{L^6}^3)
\le C (\|X(s)\|_{L^2}+\|X(s)\|_{H^1}^3)\le C
$$
for $t\in [0,T],\ \omega\in \Omega_\epsilon$. This also shows that the
integrand of the deterministic integral in \eqref{mildC-H-C-X} is in
$L_1([0,T],H)$ on $\Omega_\epsilon$ by the analyticity of the
semigroup $\ee^{-tA^2}$. Finally, since $\epsilon>0$ is arbitrary and
$P(\Omega_\epsilon)>1-\epsilon$, the statement follows.
\end{proof}

We now show that, under the stronger assumption
$\norm[\HS]{A^{\frac12}Q^{\frac12}}<\infty$, the solution $X(t)$ is
actually in $H^3$ almost surely.  In order to do this we
write, as in the previous proof, $ X(t) = Y(t) + W_A(t) $, where we
already know from Theorem \ref{Thm3.1} that $ W_A(t) $ is in $H^3$
almost surely.  The regularity of $Y$ is studied in the next
theorem.  Note that we saw in the proof of Corollary \ref{mildxy} that
\begin{align*}
Y(t) = X(t)-W_A(t)
=  \ee^{-tA^2}X_0 - \int_0^t A\ee^{-(t-s)A^2}f(X(s))\,\dd s
\end{align*}
is a weak solution of
\begin{equation}\label{tran-Cahn-Hillx}
 \dot{Y} + A^2 Y + A f(X) = 0,\quad t\in(0,T];  \quad
 Y(0) =X_0,
\end{equation}
almost surely.

\begin{thm}\label{Ynorm3}
  Assume that $\norm[\HS]{A^{\frac12}Q^{\frac12}}<\infty$ and that
  $X_0$ is $\cF_0$-measurable with values in $H^3$ satisfying
  $\norm[L_2(\Omega,H^1)]{X_0}^2 + \norm[L_4(\Omega,L_4)]{X_0}^4 \le
  \rho$.  Let $T>0$ and $\epsilon \in (0,1)$ and let $\Omega_\epsilon$
  and $K_T$ be as in Corollary \ref{coro2.5}. Let $X$ be the solution
  from Theorem \ref{pratothm} and $Y=X-W_A$.  Then $X,Y\in
  C([0,T],H)\cap L_{\infty}([0,T],H^3)$ almost surely, and, for each
$\omega\in\Omega_\epsilon$,
\begin{align}
\norm[3]{Y(t)} &\le C(\norm[3]{X_0}, \epsilon^{-1}K_T,T)
\quad \text{on}\  \Omega_\epsilon,\ t \in [0,T],\label{nyh3}\\
\norm[3]{X(t)} &\le C(\norm[3]{X_0}, \epsilon^{-1}K_T,T)
\quad \text{on}\  \Omega_\epsilon,\ t \in [0,T].\label{nxh3}
\end{align}
\end{thm}

\begin{proof}
  The continuity of $X$ is already contained in Theorem
  \ref{pratothm} and the continuity of $Y$ follows from the
  continuity of $X$ and $W_A$. To show that $X,Y\in
  L_{\infty}([0,T],H^3)$ almost surely it is enough to show
  \eqref{nyh3} and \eqref{nxh3} as $\epsilon>0$ is arbitrary and
  $P(\Omega_\epsilon)>1-\epsilon.$ Let $t\in[0,T]$ and $\omega
  \in\Omega_\epsilon$.  From Corollary \ref{coro2.5} we have
\begin{equation}\label{K1}
\norm[1]{X(t)}^2   \le  \epsilon^{-1} K_T ,\quad
 \norm[3]{W_A(t)} \le \epsilon^{-1} K_T.
\end{equation}
We take seminorms in
\begin{equation}\label{mildY}
Y(t) = \ee^{-tA^2}X_0 - \int_0^t A\ee^{-(t - s)A^2}f(X(s))\,\dd s
\end{equation}
and use \eqref{eq:analytic} to get
\begin{align*}
\snorm[3]{Y(t)}
& \le \snorm[3]{\ee^{-tA^2} X_0} + \int_0^t \snorm[3]{\ee^{-(t-s)A^2} Af(X(s))}\,\dd s\\
& = \norm{\ee^{-tA^2} A^{\frac{3}{2}}X_0} + \int_0^t \norm{A^{\frac{3}{2}}\ee^{-(t-s)A^2} Af(X(s))}\,\dd s\\
& \le \snorm[3]{X_0} + C\int_0^t (t-s)^{-\frac{3}{4}} \norm{Af(X(s))}\,\dd s.
\end{align*}
We apply \eqref{Deltaf} to $\norm{Af(X(s))}=\norm{\Delta f(X(s))} $ to get
\begin{align*}
\snorm[3]{Y(t)}
& \le \snorm[3]{X_0} + C \int_0^t (t-s)^{-\frac{3}{4}}
\big( 1+ \norm[1]{X(s)}^2\big) \norm[3]{X(s)}\,\dd s\\
& \le \snorm[3]{X_0} + C \int_0^t (t-s)^{-\frac{3}{4}}
\big( 1+ \norm[1]{X(s)}^2\big) \big(\norm[3]{Y(s)}+\norm[3]{W_A(s)}\big) \,\dd s.
\end{align*}
Since $(I-P)Y(t)=(I-P)X_0$ is constant, we get the same bound for the
norm $\norm[3]{Y(t)}$.
Using also \eqref{K1} gives
\begin{align*}
\norm[3]{Y(t)}
&\le \norm[3]{X_0} + C \int_0^t (t-s)^{-\frac{3}{4}}
\big(1+\epsilon^{-1}K_T\big)\big(\norm[3]{Y(s)} + \epsilon^{-1}K_T\big)\,\dd s
\\&
\le \norm[3]{X_0} + C
\epsilon^{-1}K_T\big(1+\epsilon^{-1}K_T\big)T^{\frac{1}{4}}
\\&\quad
+ C \big(1+\epsilon^{-1}K_T\big)\int_0^t (t-s)^{-\frac{3}{4}} \norm[3]{Y(s)}\,\dd s.
\end{align*}
Applying Gronwall's Lemma \ref{GWlem} with $\alpha = 1,\,\beta = \frac{1}{4}$ and
\begin{align}\label{A1B}
A = \norm[3]{X_0} + C
\epsilon^{-1}K_T\big(1+\epsilon^{-1}K_T\big),\
B = C \big(1+\epsilon^{-1}K_T\big),
\end{align}
gives
\begin{align*}
\norm[3]{Y(t)} \le A C(B,T) = C(\norm[3]{X_0},\epsilon^{-1}K_T,T), \quad t \in [0,T].
\end{align*}
The bound for $\norm[3]{X(t)}$ then follows in view of \eqref{K1}.
\end{proof}

The constant $C(\norm[3]{X_0},\epsilon^{-1}K_T,T)$ grows rapidly with
$\epsilon^{-1}K_T$ and $T$. Hence, it is important that $K_T$ grows
only quadratically with $T$. Also note that the proof of Theorem
\ref{Ynorm3} shows that under the assumptions of the theorem, in fact,
$f(X(t))\in D(A)$ almost surely and $\|Af(X(t))\|<\infty$ almost
surely for $t\in [0,T]$. Therefore, $X$ satisfies a more strict (in
comparison to \eqref{mildC-H-C-X}) mild form of \eqref{C-H-C-X}:
\begin{align*}
X(t) =  \ee^{-tA^2}X_0 - \int_0^t \ee^{-(t-s)A^2}Af(X(s))\,\dd s
+ \int_0^t \ee^{-(t-s)A^2}\,\dd W(s).
\end{align*}

\section{Error estimates}
\subsection{The linear deterministic Cahn-Hilliard equation}

Consider the linear Cahn-Hilliard equation
\begin{equation}\label{A1}
\dot{u} +A v = 0, \
v - Au - f = 0, \ t>0; \  u(0) = u_0,
\end{equation}
where $f$ is a function of $x,t$, and  the corresponding finite element problem
\begin{equation}\label{A2}
 \dot{u}_{h} +A_h v_h = 0, \ 
 v_h - A_hu_h - P_h f = 0, \ t>0; \ u_h(0) = P_h u_0.
\end{equation}

We have the following error estimate.  We will later use this for
fixed $\omega\in\Omega_\epsilon$ with $f$ replaced by $f(X)$ and $u$
by the solution $Y$ of \eqref{tran-Cahn-Hillx}. The error estimate
differs from the corresponding error estimates in \cite{Stig,Ali} in that
it contains no time derivative. This is important since $Y$ has
limited temporal regularity.

\begin{thm}\label{Athm1}
  Assume that $ u,v $ and $ u_h,v_h $ are weak solutions of \eqref{A1}
  and \eqref{A2}, respectively. Then, for $ t\ge0$ and $h\in(0,\frac12]$, we
  have
\begin{align}\label{A3}
\norm{u_h(t) - u(t)}
\le C h^2 \Big( |\log(h)| \max_{0\le s \le t} \snorm[2]{u(s)}
 + \Big( \int_0^t \snorm[2]{v(s)}^2\,\dd s \Big)^{\frac{1}{2}}
\Big).
\end{align}
\end{thm}

\begin{proof}
The weak forms of \eqref{A1} and \eqref{A2} are
\begin{equation}\label{A4}
\begin{aligned}
& \inner{\dot{u}}{\varphi_1}
+\inner{\nabla  v}{\nabla  \varphi_1} = 0 && \forall \varphi_1 \in H^1,\\
& \inner{v}{\varphi_2} - \inner{\nabla   u}{\nabla   \varphi_2} -
\inner{f}{\varphi_2} = 0
&& \forall \varphi_2 \in H^1,\\
& u(0) = u_0,
\end{aligned}
\end{equation}
and
\begin{equation}\label{A5}
\begin{aligned}
& \inner{\dot{u}_{h}}{\varphi_{h,1}} +\inner{\nabla   v_h}{\nabla
  \varphi_{h,1}} = 0
&& \forall \varphi_{h,1} \in S_h,\\
& \inner{v_h}{\varphi_{h,2}} - \inner{\nabla   u_h}{\nabla
  \varphi_{h,2}} - \inner{f}{\varphi_{h,2}} = 0
&& \forall \varphi_{h,2} \in S_h,\\
& u_h(0) = P_h u_0.
\end{aligned}
\end{equation}
Let $P_h$ and $R_h$ be as in \eqref{P_h} and \eqref{R_h} and set
\begin{align}
& e_u = u_h - u = (u_h -P_h u) + (P_h u - u) = \theta_u + \rho_u, \label{A6}\\
& e_v = v_h - v = (v_h -R_h v) + (R_h v - v) = \theta_v + \rho_v \label{A7}.
\end{align}
We want to compute
\begin{equation}\label{Err}
\norm{e_u} \le \norm{\theta_u} + \norm{\rho_u}.
\end{equation}
In \eqref{A4} choose $ \varphi_1 = \varphi_{h,1} $ and $ \varphi_2 =
\varphi_{h,2} $ and subtract the first two equations of \eqref{A4}
from the corresponding equations in \eqref{A5} to get
\begin{equation*}
\begin{aligned}
& \inner{\dot{e}_u}{\varphi_{h,1}} +\inner{\nabla   e_v}{\nabla
  \varphi_{h,1}} = 0
&& \forall \varphi_{h,1} \in S_h,\\
& \inner{e_v}{\varphi_{h,2}} - \inner{\nabla   e_u}{\nabla
  \varphi_{h,2}} = 0
&& \forall \varphi_{h,2} \in S_h.
\end{aligned}
\end{equation*}
Hence, by \eqref{A6} and \eqref{A7},
\begin{equation*}
\begin{aligned}
& \inner{\dot{\theta}_u}{\varphi_{h,1}}
+\inner{\nabla   \theta_v}{\nabla   \varphi_{h,1}}
= -\inner{\dot{\rho}_u}{\varphi_{h,1}} - \inner{\nabla
  \rho_v}{\nabla   \varphi_{h,1}}
&& \forall \varphi_{h,1} \in S_h,\\
& \inner{\theta_v}{\varphi_{h,2}}
- \inner{\nabla   \theta_u}{\nabla   \varphi_{h,2}}
= - \inner{\rho_v}{\varphi_{h,2}}  +  \inner{\nabla   \rho_u}{\nabla
  \varphi_{h,2}}
&& \forall \varphi_{h,2} \in S_h.
\end{aligned}
\end{equation*}
By the definitions of $ P_h $ and $ R_h $ we have
\begin{equation*}
\begin{aligned}
 \inner{\dot{\rho}_u}{\varphi_{h,1}}
&= \inner{P_h \dot{u} - \dot{u}}{\varphi_{h,1}} = 0
&& \forall \varphi_{h,1} \in S_h,\\
 \inner{\nabla   \rho_v}{\nabla  \varphi_{h,1}}
&= \inner{\nabla  R_h v - v}{\nabla   \varphi_{h,1}} = 0
&& \forall \varphi_{h,2} \in S_h,
\end{aligned}
\end{equation*}
so that
\begin{equation*}
\begin{aligned}
& \inner{\dot{\theta}_u}{\varphi_{h,1}}
+\inner{\nabla   \theta_v}{\nabla   \varphi_{h,1}} = 0
&& \forall \varphi_{h,1} \in S_h,\\
& \inner{\theta_v}{\varphi_{h,2}}
- \inner{\nabla   \theta_u}{\nabla   \varphi_{h,2}}
= - \inner{P_h\rho_v}{\varphi_{h,2}}
+  \inner{\nabla   R_h\rho_u}{\nabla \varphi_{h,2}}
&& \forall \varphi_{h,2} \in S_h.
\end{aligned}
\end{equation*}
In the second equation we  set $ \varphi_{h,2} = A_h \varphi_{h,1} $ to get
\begin{equation*}
\inner{\nabla  \theta_v}{\nabla  \varphi_{h,1}}
= \inner{A_h^2 \theta_u}{\varphi_{h,1}}
- \inner{A_h P_h \rho_v}{\varphi_{h,1}}
+\inner{A_h^2 R_h  \rho_u}{\varphi_{h,1}}.
\end{equation*}
Inserting this into the first equation gives
\begin{align*}
\inner{\dot{\theta}_u}{\varphi_{h,1}}
+ \inner{A_h^2 \theta_u}{\varphi_{h,1}}
= \inner{A_h P_h \rho_v}{\varphi_{h,1}} - \inner{A_h^2 R_h  \rho_u}{\varphi_{h,1}},
\end{align*}
so the strong form is
\begin{equation*}
\dot{\theta}_u + A_h^2 \theta_u
= A_h P_h \rho_v - A_h^2 R_h \rho_u,\quad t > 0;
\quad \theta_u(0) = 0,
\end{equation*}
with solution
\begin{align*}
\theta_u(t)
= \int_0^t \ee^{-(t - s)A_h^2}A_h P_h \rho_v(s)\,\dd s
- \int_0^t \ee^{-(t - s)A_h^2}A_h^2 R_h \rho_u(s)\,\dd s.
\end{align*}
Taking norms here gives
\begin{equation}\label{A11}
\begin{split}
\norm{\theta_u(t)}
& \le  \Norm{\int_0^t \ee^{-(t - s)A_h^2}A_h P_h \rho_v(s)\,\dd s}
\\&\quad + \Norm{\int_0^t \ee^{-(t - s)A_h^2}A_h^2 R_h \rho_u(s)\,\dd s}
 = I + II.
\end{split}
\end{equation}
For $I$ we define
\begin{equation*}
w_h (t) = \int_0^t\ee^{-(t - s)A_h^2} P_h \rho_v(s)\,\dd s,
\end{equation*}
which satisfies the equation
\begin{align*}
\dot{w}_h + A_h^2 w_h = P_h \rho_v, \quad t>0;\quad  w_h(0) = 0.
\end{align*}
We multiply by  $ \dot{w}_h $ to get
\begin{equation*}
\norm{\dot{w}_h}^2 + \frac{1}{2} \frac{\dd}{\dd t} \norm{A_h w_h}^2
= \inner{P_h \rho_v}{\dot{w}_h}
\le \norm{\rho_v}\norm{\dot{w}_h}
\le \frac{1}{2}\norm{\rho_v}^2 + \frac{1}{2}\norm{\dot{w}_h}^2,
\end{equation*}
so that
\begin{align*}
\norm{\dot{w}_h}^2 + \frac{\dd}{\dd t} \norm{A_h w_h}^2 \le \norm{\rho_v}^2.
\end{align*}
Integration and ignoring $ \int_0^t \norm{\dot{w}_h(s)}^2\,\dd s $
leads to
\begin{equation*}
\Norm{A_h \int_0^t \ee^{-(t-s)A_h^2}P_h \rho_v(s)\,\dd s}
= \norm{A_h w_h(t)}
\le \Big( \int_0^t \norm{\rho_v(s)}^2\,\dd s \Big)^{\frac{1}{2}},
\end{equation*}
where, from \eqref{chc2:Rherror},
\begin{align*}
\norm{\rho_v} = \norm{(R_h - I)v} \le Ch^2 \snorm[2]{v}.
\end{align*}
Hence,
\begin{equation}\label{AA15}
\Norm{A_h \int_0^t \ee^{-(t-s)A_h^2}P_h \rho_v(s)\,\dd s}
\le Ch^2 \Big( \int_0^t \snorm[2]{v(s)}^2\,\dd s \Big)^{\frac{1}{2}}.
\end{equation}
For the term $II$ we use
\begin{align*}
R_h \rho_u = R_h(P_h u - u) = P_h u - R_h u = P_h(u - R_h u).
\end{align*}
Then
\begin{align*}
&\Norm{\int_0^t A_h^2 \ee^{-(t-s)A_h^2}R_h \rho_u(s)\,\dd s}
\le \int_0^t \norm{A_h^2 \ee^{-(t-s)A_h^2}P_h (u(s) - R_h u(s))}\,\dd s\\
& \qquad \le \int_0^t \norm{A_h^2 \ee^{-(t-s)A_h^2}P_h}\,\dd s
\max_{0\le s \le t}\norm{u(s) - R_h u(s)}.
\end{align*}
Here we use $\norm{A_h} \le Ch^{-2}$ from \eqref{x} and
\eqref{eq:analytich} to get
\begin{equation*}
\begin{split}
& {\int_0^t \norm{A_h^2 \ee^{-(t-s)A_h^2}P_h }\,\dd s}
 = \int_0^{h^4} \norm{A_h}^2\norm{\ee^{-{s}A_h^2}}\,\dd {s}
+ \int_{h^4}^t \norm{A_h^2 \ee^{-{s}A_h^2}}\,\dd s \\
&\qquad \le  Ch^{-4} h^4 + C\int_{h^4}^t {s}^{-1}\ee^{-c{s}}\,\dd {s}
 \le  C(1+\log(1/h))\le C|\log(h)|
\end{split}
\end{equation*}
for $h\in(0,\frac12]$. Hence, by \eqref{chc2:Rherror}, we have
\begin{equation}\label{AA16}
\Norm{\int_0^t A_h^2 \ee^{-(t-s)A_h^2}R_h \rho_u(s)\,\dd s}
\le Ch^2 |\log(h)|\max_{0\le s \le t} \snorm[2]{u(s)}.
\end{equation}
Inserting \eqref{AA15} and \eqref{AA16} into \eqref{A11} gives
\begin{equation}\label{theta}
\norm{\theta_u(t)}
\le Ch^2 \Big\{ \Big( \int_0^t \snorm[2]{v(s)}^2\,\dd s \Big)^{\frac{1}{2}}
	+ |\log(h)|  \max_{0\le s \le t} \snorm[2]{u(s)} \Big\}.
\end{equation}
Finally, by the best approximation property of $P_h$,
\begin{equation}\label{rho}
\norm{\rho_u(t)}
= \norm{P_h u - u}
\le \norm{R_hu - u}
\le Ch^2 \snorm[2]{u(t)}.
\end{equation}
Inserting \eqref{theta} and \eqref{rho} into \eqref{Err} gives the desired
result \eqref{A3}.
\end{proof}

The following regularity estimate for the
linear Cahn-Hilliard equation \eqref{A1} is proved by an elementary
energy argument.

\begin{lem}\label{Alemma1}
Assume that $ u,v $ are weak solutions of \eqref{A1}. Then
\begin{align*}
\snorm[2]{u(t)}^2 + \int_0^t \snorm[2]{v(s)}^2\,\dd s
\le \snorm[2]{u_0}^2 + \int_0^t \snorm[2]{f(s)}^2 \,\dd s.
\end{align*}
\end{lem}

\subsection{Error estimate for the stochastic Cahn-Hilliard equation}
In the next theorem we prove an error estimate for the nonlinear
Cahn-Hilliard-Cook equation.

\begin{thm}\label{thm5.3}
  Assume that $\norm[\HS]{A^{\frac12}Q^{\frac12}}<\infty$ and that
  $X_0$ is $\cF_0$-measurable with values in $H^3$ satisfying
  $\norm[L_2(\Omega,H^1)]{X_0}^2 + \norm[L_4(\Omega,L_4)]{X_0}^4 \le
  \rho$.  Let $T>0$, $\epsilon \in (0,1)$, and let
  $\Omega_\epsilon\subset\Omega$ and $K_T$ be as in Corollary
  \ref{coro2.5}. If $X$ is the weak solution of \eqref{C-H-C-X} and
  $X_h$ is the solution of \eqref{FEC-H-C-X}, then, for
  $h\in(0,\frac12]$,
\begin{align*}
\norm{X_h(t) - X(t)}
\le C(\norm[3]{X_0},\epsilon^{-1}K_T,T )
h^2 \abs{\log(h)},
\quad \text{on $\Omega_\epsilon$, }t \in [0,T].
\end{align*}
\end{thm}

The constant $C(\norm[3]{X_0},\epsilon^{-1}K_T,T )$ grows rapidly
with $\epsilon^{-1}K_T$ and $T$ due to the use of Gronwall's lemma in
the proof.

\begin{proof}
Let $\omega\in\Omega_\epsilon$ be fixed.
Set
\begin{equation}\label{X(t)}
X(t) = Y(t) + W_A(t),
\end{equation}
where $W_A(t)$ is the stochastic convolution \eqref{convA}
and $ Y(t) $ is the weak solution \eqref{mildY} of
\eqref{tran-Cahn-Hillx}.  Also set
 \begin{equation}\label{X_h(t)}
 X_h(t) = Z_h(t) + W_{A_h}(t),
\end{equation}
where $ W_{A_h}(t) $ is the stochastic convolution \eqref{convAh}
and
\begin{equation}\label{mildYhat}
 Z_h(t) = \ee^{-tA_h^2}P_h X_0 - \int_0^t \ee^{-(t - s)A_h^2}A_hP_hf(X_h(s))\,\dd s.
\end{equation}
Finally, let
\begin{equation}\label{Y_h}
Y_h(t) = \ee^{-tA_h^2} P_h X_0 - \int_0^t \ee^{-(t - s)A_h^2} A_h P_h f(X(s))\,\dd s.
\end{equation}
We subtract \eqref{X(t)} from \eqref{X_h(t)}
and take norms,
\begin{equation}\label{X_h}
\norm{X_h - X}
\le \norm{W_{A_h} - W_A} + \norm{Y_h - Y} + \norm{Z_h - Y_h}.
\end{equation}
We must compute the three norms on the right-hand side.

First we compute $\norm{W_{A_h}(t) - W_A(t)}$.  Since
$\norm[\HS]{A^{\frac{1}{2}}Q^{\frac{1}{2}} } < \infty$, we have that
$\norm[\HS]{Q^{\frac{1}{2}}} < \infty$, see \eqref{KQ}, and hence, by Theorem
\ref{Thm3.2} and Chebyshev's inequality, we get
\begin{align*}
\norm{W_{A_h}(t) - W_A(t)}
& \le \epsilon^{-\frac{1}{2}}
\big(\IE[\norm{W_{A_h}(t) - W_A(t)}^2]\big)^{\frac{1}{2}}\\
& \le \epsilon^{-\frac{1}{2}} C h^2  \abs{\log(h)}
\norm[\HS]{Q^{\frac{1}{2}}}
\le C(\epsilon^{-1}K_Q)^{\frac12}h^2  \abs{\log(h)},
\end{align*}
where $K_Q$ is as in Theorem \ref{lyap-thm}.
Since $K_Q\le K_T$, we conclude
\begin{equation}\label{WAh_WA}
\norm{W_{A_h}(t) - W_A(t)}
\le C(\epsilon^{-1}K_T)^{\frac12} h^2 \abs{\log(h)}.
\end{equation}

Now we consider $ \norm{Y_h(t) - Y(t)}$ and use Theorem \ref{Athm1}
to get
\begin{equation}\label{P}
\norm{Y_h(t) - Y(t)}
\le Ch^2 \Big\{ |\log(h)|\max_{0\le s \le t} \snorm[2]{Y(s)}
+ \Big(\int_0^t \snorm[2]{V(s)}^2\,\dd s\Big)^{\frac12}\Big\},
\end{equation}
where $ Y(t) $ and $ V(t) $ are the solutions of
\begin{equation*}
\dot{Y} + AV = 0, \ 
 V = AY + f(X),\ t\in(0,T]; \  Y(0) = X_0.
\end{equation*}
By using Lemma \ref{Alemma1}, \eqref{Deltaf}, and \eqref{corofour}, we get
\begin{align*}
\int_0^t \snorm[2]{V(s)}^2 \,\dd s
&\le \snorm[2]{X_0}^2 + \int_0^t \snorm[2]{f(X(s))}^2 \,\dd s
\\&
\le \norm[2]{X_0}^2 + C \int_0^t (1+\norm[1]{X(s)}^2)\norm[3]{X(s)}\,\dd s
\\&
\le  \norm[3]{X_0}^2 + CT\big(1+ (\epsilon^{-1}K_T)^{\frac{3}{2}})\big).
\end{align*}
Therefore,
\begin{equation}\label{intVt}
\int_0^t \snorm[2]{V(s)}^2 \,\dd s
\le C(\norm[3]{X_0},\epsilon^{-1}K_T,T ).
\end{equation}
Now we bound $\snorm[2]{Y(t)}$.
By Theorem \ref{Ynorm3} we have
\begin{equation}\label{Y2}
\snorm[2]{Y(t)} \le \norm[3]{Y(t)} \le C(\norm[3]{X_0}, \epsilon^{-1} K_T,T).
\end{equation}
Using \eqref{intVt} and \eqref{Y2} in \eqref{P} gives
\begin{equation}\label{Y_h(t)}
\norm{Y_h(t) - Y(t)}
\le C(\norm[3]{X_0}, \epsilon^{-1} K_T,T) h^2 \abs{\log(h)}.
\end{equation}

Finally we compute $ \norm{e_h(t)}=\norm{Z_h(t) - Y_h(t)}$. By
subtraction of \eqref{mildYhat} and \eqref{Y_h}, we obtain
\begin{align*}
\norm{e_h(t)}
& \le \int _0^t \norm{\ee^{-(t-s)A_h^2} A_h P_hP(f(X_h(s)) - f(X(s))}\,\dd s\\
& \le \int _0^t \norm{A_h^{\frac{3}{2}} \ee^{-(t-s)A_h^2}P_h} \norm{
  A_h^{-\frac{1}{2}} P(f(X_h(s)) - f(X(s))}\,\dd s,
\end{align*}
since the constant eigenmodes cancel (cf.~\eqref{chc2:cancel}).
Using \eqref{Deltauv-1} and \eqref{eq:analytich} gives
\begin{equation*}
\norm{e_h(t)}
\le C \int_0^t (t - s)^{-\frac{3}{4}}\big(1+\norm[1]{X_h(s)}^2
+ \norm[1]{X(s)}^2\big)\norm{X_h(s) - X(s)}\,\dd s.
\end{equation*}
By Corollary \ref{coro2.5} we have
\begin{align*}
 \norm{e_h(t)}
& \le C \int_0^t (t - s)^{-\frac{3}{4}} \big( 1 + \epsilon^{-1}K_T\big)
	\big( \norm{W_{A_h}(s) - W_A(s)}
\\& \quad
+ \norm{Y_h(s) - Y(s)}
+ \norm{e_h(s)} \big)\,\dd s\\
& \le C \big( 1 + \epsilon^{-1}K_T\big) T^{\frac14}
	\max_{0\le s \le T}\big( \norm{W_{A_h}(s) - W_A(s)} + \norm{Y_h(s) - Y(s)} \big)\\
& \quad	+ C \big( 1 + \epsilon^{-1}K_T\big) \int_0^t (t-s)^{-\frac{3}{4}} \norm{e_h(s)}\,\dd s.
\end{align*}
We apply Gronwall's Lemma \ref{GWlem} with $ \alpha = 1,\,\beta = \frac{1}{4}$ and
\begin{align*}
& A = C \big(1+\epsilon^{-1}K_T\big) T^{\frac{1}{4}}
\max_{0\le s \le T}\big( \norm{W_{A_h}(s) - W_A(s)} + \norm{Y_h(s) - Y(s)} \big),\\
& B = C \big(1+\epsilon^{-1}K_T \big),
\end{align*}
to get
\begin{equation}\label{Yhat-Y}
\norm{Z_h(t) - Y_h(t)}=\norm{e_h(t)} \le A C(B,T), \quad t\in[0,T].
\end{equation}
But we already obtained bounds for $ \norm{W_{A_h}(t) - W_A(t)} $ and
$ \norm{Y_h(t) - Y(t)} $ in \eqref{WAh_WA} and \eqref{Y_h(t)}. By
inserting these and \eqref{Yhat-Y} into \eqref{X_h} we get the
desired result.
\end{proof}

Since we have regularity of order $3$ on $\Omega_\epsilon$, it
  would be possible to prove convergence of order $3$ for piecewise
  quadratic finite elements.  We do not find this worth the extra
  effort.

We finally show that $X_h$ converges strongly to $X$. More precisely,
we show that $X_h(t) \to X(t)$ in $L_2(\Omega,H)$ uniformly on $[0,T]$
as $ h \to 0$.

\begin{thm}\label{lastthm}
  Assume that $\norm[\HS]{A^{\frac12}Q^{\frac12}}<\infty$ and that
  $X_0$ is $\cF_0$-measurable with values in $H^3$ satisfying
  $\norm[L_2(\Omega,H^1)]{X_0}^2 + \norm[L_4(\Omega,L_4)]{X_0}^4 \le
  \rho$ for some $\rho$.  If $X$ is the weak solution of
  \eqref{C-H-C-X} and $X_h$ is the solution of \eqref{FEC-H-C-X}, then
\begin{equation*}
\max_{t \in [0,T]}
\big( \IE[\norm{X_h(t) - X(t)}^2]\big)^{\frac{1}{2}}
\to 0
\quad \text{as}\ h \to 0.
\end{equation*}
\end{thm}

\begin{proof}
From Theorem \ref{lyap-thm} it follows that
\begin{align} \label{eq:star}
\IE\big[\norm[L_4]{X(t)}^4 \big] \le K_T,\quad
 \IE\big[\norm[L_4]{X_h(t)}^4 \big] \le K_T,\quad t \in[0,T],
\end{align}
with $K_T$ as in Corollary \ref{coro2.5}. Let $\epsilon\in(0,1)$ and
let $\Omega_\epsilon$ be as in Corollary \ref{coro2.5}.  Then
\begin{equation*}
\begin{split}
\IE\big[ \norm{X_h(t) - X(t)}^2 \big]
&\le  \int_{\Omega_\epsilon} \norm{X_h(t) - X(t)}^2 \,\dd \IP \\
&\quad + 2\int_{\Omega_\epsilon^c} \big( \norm{X_h(t)}^2 + \norm{X(t)}^2\big) \,\dd \IP.
\end{split}
\end{equation*}
Here, by H\"older's inequality and \eqref{eq:star}, we have
\begin{equation*}
\begin{split}
\int_{\Omega_\epsilon^c} \norm{X(t)}^2\,\dd \IP
& \le \Big( \int_{\Omega_\epsilon^c} 1^2\,\dd \IP \Big)^{\frac{1}{2}}
\Big( \int_{\Omega_\epsilon^c} \norm[L_4]{X(t)}^4\,\dd \IP \Big)^{\frac{1}{2}}\\
& \le \epsilon^{\frac{1}{2}} \big( \IE\big[ \norm[L_4]{X(t)}^4
  \big] \big)^{\frac{1}{2}}
\le \epsilon^{\frac{1}{2}} K_T^{\frac{1}{2}}
\end{split}
\end{equation*}
and similarly for $X_h$.  Therefore, by Theorem \ref{thm5.3},
\begin{align*}
\max_{t \in [0,T]}\big( \IE\left [\norm{X_h(t) - X(t)}^2\right ] \big)^{\frac{1}{2}}
\le C(\epsilon^{-1}K_T,T) h^2 \abs{\log (h)} + CK_T^{\frac{1}{4}} \epsilon^{\frac{1}{4}}.
\end{align*}
Since $\frac{\epsilon^{\frac{1}{4}}}{C(\epsilon^{-1}K_T,T)} \to 0$
monotonically as $\epsilon \to 0$, we may choose $\epsilon$ depending
on $h$, such that the two terms are equal.
\end{proof}

Since $ C(\epsilon^{-1}K_T,T) $ grows rapidly with $\epsilon^{-1}$, it
is not possible to obtain a rate of convergence from this proof.


\providecommand{\bysame}{\leavevmode\hbox to3em{\hrulefill}\thinspace}
\providecommand{\MR}{\relax\ifhmode\unskip\space\fi MR }
\providecommand{\MRhref}[2]{%
  \href{http://www.ams.org/mathscinet-getitem?mr=#1}{#2}
}
\providecommand{\href}[2]{#2}

\end{document}